\documentclass[12pt]{amsart}
\usepackage{times}
\DeclareMathAlphabet{\mathpzc}{OT1}{pzc}{m}{it}

\usepackage[T1]{fontenc}
\usepackage{dsfont}
\usepackage{mathrsfs}
\usepackage[colorlinks]{hyperref}
\usepackage{xcolor}
\usepackage[a4paper,asymmetric]{geometry}
\usepackage{mathscinet}
\usepackage{fullpage}
\usepackage{latexsym}
\usepackage{graphicx}
\usepackage{epstopdf}
\usepackage{amsthm}
\usepackage{amssymb}
\usepackage{amsfonts}
\usepackage{amsmath}

\def\az{\alpha}  \def\bz{\beta}
\def\cz{\chi}    \def\dz{\delta}
    
  \def\kz{\kappa}
\def\lz{\lambda} \def\mz{\mu}

\def\qz{\psi}
        \def\sz{\sigma}

\def\qd{\quad}
\def\qqd{\qquad}

\def\lt{\left}
\def\rt{\right}

\def\PP{\mathbb{P}}
\def\EE{\mathbb{E}}

\newcommand{\mathsym}[1]{{}}

\newcommand{\Be}{\begin{equation}}
\newcommand{\Ees}{\end{equation*}}
\newcommand{\Bes}{\begin{equation*}}
\newcommand{\Ee}{\end{equation}}

\def\le{\leqslant}
\def\ge{\geqslant}
\def\leq{\leqslant}
\def\geq{\geqslant}

\newtheorem{thm}{Theorem}[section]

\newtheorem{lem}[thm]{Lemma}
\newtheorem{rem}[thm]{Remark}
\newtheorem{cor}[thm]{Corollary}

\newtheorem{ass}[thm]{Assumption}

\numberwithin{equation}{section} \allowdisplaybreaks[4]

\def\prf{\medskip \noindent {\bf Proof}. }
\def\deprf{\quad $\square$ \medskip}
\def\bg{\begin}
\def\be{\begin{equation}}
\def\de{\end{equation}}
\def\ben{\begin{equation*}}
\def\den{\end{equation*}}

\def\dear{\end{eqnarray*}}
\def\lb{\label}

\def\dps{\displaystyle}

\def\ben{\bg{enumerate}}
\def\den{\end{enumerate}}

\def\d{\text{\rm d}}
\def\p{\partial}
\def\L{\mathcal{L}}
\def\C{\mathcal{C}}
\def\B{\mathcal{B}}

\def\PP{\mathbb{P}}
\def\EE{\mathbb{E}}
\def\RR{\mathbb{R}}

\def\NN{\mathbb{N}}

\begin{document}

\title[Quantitative diffusion approximation for the Neutral $r$-Alleles Wright-Fisher Model with Mutations]
{Quantitative diffusion approximation for the Neutral $r$-Alleles Wright-Fisher Model with Mutations}

\author[P.~Chen]{Peng Chen}
\address[P.~Chen]{School of Mathematics, Nanjing University of Aeronautics and
Astronautics, Nanjing 211106, China}
\email{chenpengmath@nuaa.edu.cn}

\author[J.~Xiong]{Jie Xiong}
\address[J.~Xiong]{Department of Mathematics, Southern University of Science and Technology, Shenzhen,
China}
\email{xiongj@sustc.edu.cn}

\author[L.~Xu]{Lihu Xu}
\address[L.~Xu]{1. Department of Mathematics, Faculty of Science and Technology, University of Macau, Macau S.A.R., China. 2. Zhuhai UM Science \& Technology Research Institute, Zhuhai, China}
\email{lihuxu@um.edu.mo}

\author[J.~Zheng]{Jiayu Zheng}
\address[J.~Zheng]{Faculty of Computational Mathematics and Cybernetics, Shenzhen MSU-BIT University, Shenzhen, 518172, China.}
\email{jy\_zheng@outlook.com}

\keywords{Neutral $r$-Alleles Wright-Fisher Model; Stochastic Differential Equation; Lindeberg Principle; Markov Property.  \\ Equal contributions: the authors are listed in the alpha-beta order.}
\makeatletter



\begin{abstract}
We apply a Lindeberg principle under the Markov process setting to approximate the Wright-Fisher model with neutral $r$-alleles using a diffusion process, deriving an error rate based on a function class distance involving fourth-order bounded differentiable functions. This error rate consists of a linear combination of the maximum mutation rate and the reciprocal of the population size. Our result improves the error bound in the seminal work \cite{EN77}, where only the special case $r=2$ was studied.  
\end{abstract}

\maketitle

\section{Introduction}
The one-locus, two-alleles Wright-Fisher model within a population of $N$ diploid individuals is a homogeneous Markov chain $\{ Y(n), n \in \NN \}$ where $\NN := \{0, 1, 2, \ldots \}$, characterized by a state space $\lt\{ \frac{i}{2N}: i = 0, 1, \ldots, 2N \rt\}$ and the transition probabilities given as: $\forall \ y \in \lt\{ \frac{i}{2N}: i = 0, 1, \ldots, 2N \rt\}$
\be\label{WF-D}
\PP \lt\{Y(n+1) = \frac{j}{2N} \Big| Y(n) = y \rt\} = C_{2N}^j (y^{\#})^j (1 - y^{\#})^{2N - j}.
\de
Here $C_{2N}^j$ is the standard combinatorial number, and $y^{\#}$ is the proportion adjusted by mutations. Specifically, for given mutation parameters $\mz_1, \mz_2 \in [0,1]$, we define
\begin{align*}
y^{\#} = \lt(1-\mz_1\rt)y+\mz_2(1- y),
\end{align*}
and it is easy to verify that $y^{\#} \in (0, 1)$ holds true for all 
$y \in \lt\{ \frac{i}{2N}: i = 0, 1, \ldots, 2N \rt\}$. The expression (\ref{WF-D}) implies that given $Y(n) = y$, the subsequent step in the process satisfies $2N Y(n+1) \sim \B (2N, y^{\#})$.  The seminal work \cite{EN77} studied the error bound between $\{ Y(n), n \in \NN \}$ and the stochastic process $\{X(t), t \geq 0\}$ with a state space of $[0, 1]$ governed by 
\be\lb{WF-C}
\d X(t) =\sqrt{\frac{X(t)(1-X(t))}{2N}} \d B(t) +\left[- \mz_1 X(t) + \mz_2 (1 - X(t))\right] \d t, \qd X(0) = x,
\de
where $B(t)$ represents the standard Brownian motion, and proved the following theorem by a certain clever Taylor expansion.  
\begin{thm} \label{EN}
$($\cite{EN77}$)$ Assume that $N\geq1$, $\mu_{1},\mu_{2}\in[0,1]$. Then, for any $x\in \lt\{ \frac{i}{2N}: i = 0, 1, \ldots, 2N \rt\}$ and $f\in\mathcal{C}_{b}^{6}(I)$, we have
\begin{align*}
\left|\mathbb{E}_{x}[f(Y(n))]-\mathbb{E}_{x}[f(X(n))]\right|
\leq&\frac{1}{2}\left\{\frac{\max\{\mu_{1},\mu_{2}\}}{2}\|f^{(1)}\|+\frac{\theta}{16N}\|f^{(2)}\|
+\frac{1}{216\sqrt{3}N}\|f^{(3)}\|\right\}\\
&+
\frac{1}{4}\left\{\frac{9\max\{\mu_{1}^{2},\mu_{2}^{2}\}}{2}\sum_{j=1}^{6}\|f^{(j)}\|+\frac{7}{16N^{2}}\sum_{j=2}^{6}\|f^{(j)}\|\right\},
\end{align*}
where
\begin{align*}
\theta=\frac{1+4\max\{\mu_{1},\mu_{2}\}}{1+2(\mu_{1}+\mu_{2})}.
\end{align*}
\end{thm}
Note that the spaces of the Markov chain $\{ Y(n), n \in \NN \}$ and
the diffusion process $\{X(t), t \ge 0\}$ are  $\lt\{ \frac{i}{2N}: i = 0, 1, \ldots, 2N \rt\}$ and $[0,1]$ respectively, but the relation in the previous theorem implies an error bound between the distributions of $Y(n)$ and $X(n)$ in a certain smooth Wasserstein distance, see \cite{ChMe08}.

Wright-Fisher models can be traced back to the pioneering works of \cite{F30} and \cite{W31}, where the binomial distribution was first explicitly used as a model for gene frequency in a finite population after a single generation of random mating. For an in-depth exploration of the Wright-Fisher model incorporating mutations, further insights can be found in \cite{D17,EK93,GHK21,HJT17,T12,Wan23}, and the related references therein. These references provide comprehensive discussions on various aspects of this model and its implications.

We shall consider in this paper the one-locus, $r$-alleles Wright-Fisher model with $N$ diploid individuals. There are $r$ ($r \geq 2$) alleles $A_1, A_2, \ldots, A_r$ at the given locus with $r(r+1)/2$ possible genotypes, i.e., $A_i A_j$, $1 \leq i \leq j \leq r$. The corresponding Wright-Fisher model is a homogeneous $(r-1)$-dimensional Markov chain, which will be specified with more details in next section.  We shall derive the approximating diffusion process of the Markov chain and prove a non-asymptotic error bound between   them. Our method is based on a Lindeberg principle under the Markov process setting, as $r=2$ we can derive a better error bound.    

The paper is organized as follows. In Section \ref{frame}, we formally introduce the neutral $r$-alleles Wright-Fisher model with mutations and outline the primary results regarding diffusion approximation. In Section \ref{auxiliary}, we  present crucial auxiliary lemmas, which play a crucial role in proving the main results. In Section \ref{proof}, we provide an in-depth breakdown of the proof for the main conclusion. In Appendix \ref{AppendixA}, we focus on establishing the existence and uniqueness of the solution to equation \eqref{WF-CH} below. This appendix serves as supplementary proof material reinforcing the foundational aspects of the proposed framework.

At the conclusion of this section, we introduce notations that will be frequently used in subsequent discussions. For  any vectors $x, y \in \RR^{r-1}$ with $r \ge 2$ and matrices $A, B \in \RR^{(r-1) \times (r-1)}$, we introduce the inner product $\langle \cdot, \cdot \rangle$ and $\langle \cdot, \cdot \rangle_{\rm HS}$ as follows:
\be\lb{inner}
\langle x, y \rangle := \sum_{i=1}^{r-1} x_i y_i, \qd \langle A, B \rangle_{\rm HS} := {\rm tr} (B^\top A) = \sum_{i, j =1}^{r-1} a_{ij} b_{ij}.
\de
The inner product $\langle \cdot, \cdot \rangle_{\rm HS}$ also known as the Hilbert-Schmidt inner product in $\RR^{(r-1) \times (r-1)}$. For any matrix $E = \lt( e_{ij} \rt)_{1 \leq i, j \leq r-1} \in \RR^{(r-1) \times (r-1)}$, two commonly employed norms induced by the Hilbert-Schmidt inner product are defined as:
$$
\| E \|_{\rm op} = \sup_{\| x \| \leq 1} \| E x \|, \qd \| E \|_{\rm HS} = \sqrt{\langle E, E \rangle_{\rm HS}} = \sqrt{\sum_{i, j =1}^{r-1} e_{ij}^2},
$$
where $\| \cdot \| := \sqrt{\langle \cdot, \cdot \rangle}$ is the Euclidean norm in $\RR^{r-1}$. The norms $\| \cdot \|_{\rm op}$ and $\| \cdot \|_{\rm HS}$ also known as the operator norm and the Hilbert-Schmidt norm in $\RR^{(r-1) \times (r-1)}$. According to  references \cite{D69,DNN11},  the following relation holds: $\| AB \|_{\rm HS} \leq \| A \|_{\rm op} \| B \|_{\rm HS}$ for all $A, B \in \RR^{(r-1) \times (r-1)}$. Additionally, it is valid that:
\be\lb{relation}
\| E \|_{\rm op} = \sup_{\| x\| \leq 1, \|y \| \leq 1} \lt| \langle E, x y^\top \rangle_{\rm HS} \rt| \qd \text{and} \qd \| E \|_{\rm op} \leq \| E \|_{{\rm HS}} \leq \sqrt{r-1} \| E \|_{\rm op}.
\de

Let $\C_b (\RR^{r-1})$ with $r \geq 2$ denote the set of all bounded continuous functions $f : \RR^{r-1} \to \RR$, and $\C_b^k (\RR^{r-1})$ with $k \geq 1$ represents the set of all bounded $k$-th order continuously differentiable functions. For any $f \in \C^2_b (\RR^{r-1})$, denote by $\nabla f(x) \in \RR^{r-1}$ and $\nabla^2 f(x) \in \RR^{(r-1) \times (r-1)}$, the gradient and the Hessian matrix of $f (x)$ respectively. Moreover, in order to state our main results,  we introduce additional notation.  For any $f \in \C^n (\RR^{r-1})$ and $1 \leq m \leq n$, we denote  $\nabla^m f (x)$  as the $m$-th order differentiation of $f(x)$, defined as:
$$
\nabla^m f (x) = \lt\{ \frac{\p^m}{\p x_{k_1} \cdots \p x_{k_m}} f (x) \ \Big| \ x \in \RR^{r-1}, \ k_1, \ldots, k_m \in \{1, \ldots, r-1\} \rt\}.
$$
Define the $L^2$ norm of $\nabla^m f (x)$ as
\be\lb{norm}
\big\| \nabla^m f (x) \big\|_2 := \lt[ \sum_{k_1, \ldots, k_m =1}^{r-1} \lt( \frac{\p^n}{\p x_{k_m} \cdots \p x_{k_1}} f (x) \rt)^2 \rt]^{1/2}.
\de
According to the definition of $\| \cdot \|_2$ norm given in (\ref{norm}), it holds that
$$
\| \nabla f \|_2 = \| \nabla f\|, \qd \text{and} \qd \| \nabla^2 f \|_2 = \| \nabla^2 f \|_{\rm HS}.
$$
This implies the $\| \cdot \|_2$ norm is consistent with the Euclidean norm (specifically for $m=1$), and similarly, the Hilbert-Schmidt norm (for $m=2$).

\section{Setting and main results}\label{frame}

We consider in this paper the one-locus, $r$-alleles Wright-Fisher model with $N$ diploid individuals. There are $r$ ($r \geq 2$) alleles $A_1, A_2, \ldots, A_r$ at the given locus with $r(r+1)/2$ possible genotypes, i.e., $A_i A_j$, $1 \leq i \leq j \leq r$. The Wright-Fisher model is a homogeneous Markov chain $\big\{Y (n) = (Y_1 (n), Y_2 (n), \ldots, Y_{r-1} (n)), n \in \NN \big\}$ with state space
$$
I_{2N} = \lt\{ \bigg( \frac{\az_{1}}{2N}, \ldots, \frac{\az_{r-1}}{2N} \bigg): \text{$\az_{i} \in \{0, 1, \ldots, 2N \}$ for $1 \leq i \leq r-1$ and $\sum_{i=1}^{r-1} \az_{i} \leq 2N$}\rt\}.
$$
The transition probability of $Y (n)$ is given by the following multinomial formula
\be\lb{WF-D-r}
\PP \lt( Y {(n+1)} = \Big(\frac{\az_{1}}{2N}, \ldots, \frac{\az_{r-1}}{2N} \Big) \ \Big| \ Y (n) = y \rt) = \frac{(2N) !}{(\az_{1})! \cdots (\az_{r-1})!} \prod_{i=1}^{r-1} \lt( y_{i}^{\#} \rt)^{\az_{i}},
\de
where $\az_{i} \in \{1, 2, \ldots, 2N\}$ satisfies $\sum_{i=1}^{r-1} \az_{i} \leq 2N$ and $y \in I_{2N}$. Here $y_{i}^{\#}$ represents the proportion of the allele $A_i$ adjusted by $u_{ij}$,  defined as
\be\lb{muta-r}
y_i^{\#} = y_i \lt( 1 - \sum_{j=1}^r u_{ij} \rt) + \sum_{j =1}^r u_{ji} y_j,
\de
where $y_r = 1- \sum_{i=1}^{r-1} y_i$ and $u_{ij} \in [0, 1]$ are the mutation parameters denoting the mutation probabilities from allele $A_i$ to $A_j$. The transition probability, as given in (\ref{WF-D-r}), implies that when $Y(n) = (y_1, \ldots, y_{r-1}) \in I_{2N}$, the distribution of $Y (n+1)$ satisfies the multinomial distribution
$$
\Big(Y_1 (n+1), \ldots, Y_{r-1} (n+1) \Big) \sim (2N)^{-1} {\rm multinomial} \lt( 2N, \ \big( y_1^{\#}, \ldots, y_{r-1}^{\#} \big) \rt).
$$
Regarding the mutation mechanism within this model, we outline the following essential assumptions.
\begin{ass}\lb{ass}
\
\begin{itemize}
  \item[(i).] The two genes carry by an individual mutate independently, and the mutation parameters satisfy $u_{ii} = 0$ for each $i \in \{1, \ldots, r-1\}$.
  \item[(ii).] For each $i \in \{1, \ldots, r-1\}$, $u_{ki} = u_{ri}$ holds for all $k \neq i$.
\end{itemize}
\end{ass}

According to \cite[Chapter 10]{EK09}, the $r$-alleles Wright-Fisher model $\{Y(n), n \in \NN\}$ can be typically approximated by a $(r-1)$-dimension diffusion process $\{X (t), t \geq 0\}$ defined in the state space $I$ as
$$
I = \lt\{ x = ( x_1, \ldots, x_{r-1} ) \in \RR^{r-1}: \text{$0 \leq x_i \leq 1$ for $1 \leq i \leq r-1$ and $\sum_{i=1}^{r-1} x_i \leq 1$}\rt\}.
$$
Define by $\mathcal{L}$ the second order elliptic operator as
\be\lb{oper}
\mathcal{L} = \frac{1}{4N} \sum_{i, j = 1}^{r-1} a_{ij} (x) \frac{\p^2}{\p x_i \p x_j} + \sum_{i=1}^{r-1} b_i (x) \frac{\p}{\p x_i}, \qd \forall x \in I,
\de
where
\begin{align}
a_{ij} (x) &= x_i (\dz_{ij} - x_j), \qd x \in I, \label{arr} \\
b_i (x) &= - x_i \sum_{j=1}^r u_{ij} + \sum_{j=1}^r x_j u_{ji}, \qd x \in I, \label{mul}
\end{align}
 and $x_{r}=1-\sum_{i=1}^{r-1}x_{i}\in[0,1]$. For simplicity, we denote the matrix $A(x) = \lt( a_{ij} (x) \rt)_{1 \leq i, j \leq r-1} \in \RR^{(r-1) \times (r-1)}$ and the vector $b(x) = \lt( b_i (x) \rt)_{1 \leq i \leq r-1} \in \RR^{r-1}$. The condition $x\in I$ ensures that $A (x)$ is a positive semi-definite matrix. Consequently,  there exists a real-valued matrix $\sz (x)$ such that $\sz (x) \sz^{\top} (x) = A (x)$. The diffusion process $X (t)$ in $I$ corresponding to the generator $\mathcal{L}$ is defined as
\be\lb{WF-CH}
\d X (t) = \frac{1}{\sqrt{2 N}} \sz (X(t)) \d B(t) + b(X(t)) \d t, \qd X (0) = x \in I.
\de
By the same argument as presented in the proof outlined in \cite{Eth76}, we can deduce the existence and uniqueness of the martingale solutions to the above equation. This signifies the establishment of solutions' existence and uniqueness for the martingale problem governed by $\mathcal{L}$, starting at any $x\in I$. In particular, when $r=2$, we can obtain the existence and uniqueness of the strong solution with the help of \cite[Theorem 1]{YW71}. For the convenience of readers, we offer the detailed proof in the Appendix \ref{AppendixA}.

From the definitions of $X (t)$ and $Y (n)$, it is clear that the state spaces adhere to $I_{2N} \subset I$.  The main purpose of our work is to estimate the error in the approximation between the diffusion process $X (t)$
and the Markov chain $Y (n)$. The key approach involves employing the Lindeberg principle in the Markov process setting, initially introduced by \cite{Lin22}  to establish the classical central limit theorem. This principle has found extensive applications and generalisations across various research problems, as evidenced by its utilization in \cite{CSZ16,Cha06, CLX22, CSX20, CX19,CCK14,CCK17,KM11,TV11} and the references therein.

Now, we state our main theorem concerns estimating the error involved in the diffusion approximation of the $r$-alleles Wright-Fisher model.

\begin{thm}\lb{main1}
Suppose that Assumption \ref{ass} holds and $r\geq3$. For each $n \in \NN$, $f \in \mathcal{C}_b^4 (I)$ and $x \in I_{2N}$, it holds that
\begin{align*}
& \lt| \EE_x \lt[ f (X (n)) \rt] - \EE_x \lt[ f( Y (n)) \rt] \rt| := \lt| \EE \lt[ f (X (n)) \big| X(0) = x \rt] - \EE \lt[ f( Y (n)) \big| Y(0) = x \rt] \rt| \\
& \qd \leq \sum_{m=1}^4 C_m \lt\{ \frac{1 - \exp \lt[-n \lt( \min_{1 \leq k_1, \ldots, k_m \leq r-1} \lz_{k_1, \ldots, k_m} \rt) \rt]}{1- \exp \lt[ - \min_{1 \leq k_1, \ldots, k_m \leq r-1} \lz_{k_1, \ldots, k_m} \rt]} \rt\} \sup_{x \in I}\Big\|  \nabla^m f (x) \Big\|_2.
\end{align*}
The constants $\lz_{k_1, \ldots, k_m}$ depend on the mutation parameters, i.e.,
\be\lb{lz}
\lz_{k_1, \ldots, k_m} = \sum_{j=1}^r \sum_{i=1}^m u_{k_i j} + \sum_{i=1}^m u_{r k_i} + \frac{m(m - 1)}{4N}, \qd m = 1, 2, 3, 4.
\de
and $C_m$ are define as
\begin{align*}
C_1 &= \frac{1}{2} b^* u^*,\\
C_2 &= \frac{\sqrt{r-2}}{8N\sqrt{r-1}}\left(2u^{*}+\frac{1}{2N}\right)+\frac{b^{*}}{2}\left(b^{*}+\frac{\sqrt{5}}{4N}\right)+\frac{(r-1)u^{*}}{\sqrt{2}}\left(\frac{1}{N}+u^{*}\right),\\
C_3 &= \frac{\sqrt{r-2}}{8N\sqrt{r-1}}\left(b^{*}+\frac{b^{*}}{\sqrt{r-1}}+\frac{3\sqrt{2}}{4N}\right)+\frac{1}{6}(r-1)^{\frac{3}{2}}\left[\frac{1}{32N^{2}}+\frac{3u^{*}}{8N}+(u^{*})^{3}\right],\\
C_4 &=\frac{r-2}{32N^{2}(r-1)^{\frac{3}{2}}},
\end{align*}
where $i^* := \arg\max\limits_{i} \lt( \sum_{j=1}^r u_{ij} + u_{ri}\rt)$ and
$$
b^* = \lt[ \lt( \sum_{j=1}^r u_{i^* j} \rt)^2 + \sum_{i \neq i^*} u_{ri}^2 \rt]^{1/2}, \qd u^* = \dps\max_{1 \leq k \leq r-1} \bigg( \sum_{j=1}^r u_{kj} + u_{rk} \bigg).
$$
\end{thm}
\begin{rem}
The state space of $Y(n)$ is discrete, while that of $X(n)$ is continuous, the bound in the theorem still gives a bound between
the laws of $Y(n)$ and $X(n)$ in a certain smooth Wasserstein distance, see \cite{ChMe08}.  
\end{rem}

When $r=2$, by the same argument as the proof of Theorem \ref{main1}, we immediately obtain the following corollary, 
which outperforms the result in \cite{EN77}. 

\begin{cor}\label{cor}
Keep the same notation and assumption as above. Let $r=2$. Then, for any $x\in I_{2N}$ and $f\in\mathcal{C}_{b}^{4}(I)$, we have
\begin{align*}
\lt| \EE_x \lt[ f (X (n)) \rt] - \EE_x \lt[ f( Y (n)) \rt] \rt|\leq \sum_{m=1}^4 \tilde{C}_m \lt\{ \frac{1 - \exp \lt[-n\tilde{\lambda}_{m}\rt]}{1- \exp \lt[ -\tilde{\lambda}_{m}\rt]} \rt\}\sup_{x \in I}\Big|f^{(m)}(x) \Big|,
\end{align*}
where $\tilde{\lambda}_{m}=\frac{m(m-1)}{4N}+m(u_{12}+u_{21})$,
\begin{align*}
\tilde{C}_1=\max\{u_{12},u_{21}\}(u_{12}+u_{21}),
\end{align*}
\begin{align*}
\tilde{C}_2=\frac{1}{64N^{2}}+\frac{u_{12}+u_{21}}{16N}+\max\{u_{12}^{2},u_{21}^{2}\}+\frac{5\max\{u_{12},u_{21}\}}{8N},
\end{align*}
\begin{align*}
\tilde{C}_3=\frac{1}{48N^{2}}+\frac{\max\{u_{12},u_{21}\}}{8N}+\frac{\max\{u_{12}^{3},u_{21}^{3}\}}{6},
\end{align*}
\begin{align*}
\tilde{C}_4=\frac{1}{512N^{2}}.
\end{align*}
\end{cor}

\begin{rem}
Comparing Corollary \ref{cor} and Theorem \ref{EN}, disregarding the influence of constants, both exhibit a convergence rate of $\frac{1}{N}$. However, Theorem \ref{EN} requires $f\in\mathcal{C}_{b}^{6}(I)$, whereas Corollary \ref{cor} demands only $f\in\mathcal{C}_{b}^{4}(I)$. Consequently, the result obtained from Corollary \ref{cor} is better than that of Theorem \ref{EN}.
\end{rem}

\section{Auxiliary Lemmas}\label{auxiliary}

To proceed with the remaining claims, we require the following preliminary steps.  Using the notations given in (\ref{inner}), we can rewrite the operator of (\ref{oper}) as
\be\lb{oper2}
\L f (x) = \frac{1}{4N} \langle A(x), \nabla^2 f (x) \rangle_{\rm HS} + \langle b(x), \nabla f (x) \rangle, \qd \forall f \in \C^2_b (I), \ x \in I,
\de
where $A(x) = (a_{ij} (x))_{1 \leq i, j \leq r-1}$ and $b(x) = (b_i (x))_{1 \leq i \leq r-1}$. Let $X (t)$ be the diffusion process in $I$ with generator $\mathcal{L}$ given by (\ref{oper}), which is the solution to (\ref{WF-CH}). Then, by \cite[Lemma 9.3.1]{Nor72}, we have
\begin{align}\label{opde}
\lim_{t \to 0+} \lt\| \frac{P_{t} f - f}{t} - \L f \rt\| = 0,
\end{align}
where $P_t$ is the semigroup of the process $X (t)$, i.e., $P_{t}f(x)= \mathbb{E} \lt[ f(X (t)) | X (0) = x \rt]$ for all $x \in I$ and $t \geq 0$. In the following, if there is no ambiguity, the conditional expectation $\EE \lt[ \cdot | X (0) = x \rt]$ is abbreviated as $\EE_x \lt[ \cdot \rt]$. There exists a close relationship between equation (\ref{opde}) and the Kolmogorov backward equation, which becomes apparent in the subsequent proofs. Based on  (\ref{opde}) and the Chapman-Kolmogorov equation, it holds that
\be\label{power}
\frac{\partial^{n}}{\partial t^{n}} ( P_{t} f) (x) = (P_{t} \L^{n} f) (x), \qd \text{for all $x \in I$, $t \geq 0$ and $f \in \mathcal{C}^{2n} (I)$ with $n \geq 1$}.
\de
We begin by providing estimates for the higher-order derivatives of  $P_{t} f$.

\begin{lem}\label{regular}
Suppose that Assumption \ref{ass} holds and $f \in \mathcal{C}_{b}^{4} (I)$. For each $m \in \{1, 2, 3, 4\}$ and $k_1, \ldots, k_m \in \{1, 2, \ldots , r - 1 \}$, we have
\be\lb{semi}
\bigg| \sup_{x \in I} \frac{\p^m}{\p x_{k_1} \cdots \p x_{k_m}} P_t f (x) \bigg| \leq \bigg| \sup_{x \in I} \frac{\p^m}{\p x_{k_1} \cdots \p x_{k_m}} f (x) \bigg| \exp{ \bigg[ -t \lz_{k_1, \ldots, k_m} \bigg]},
\de
where $\lz_{k_1, \ldots, k_m}$ are defined in (\ref{lz}), which are the constants depending on the mutation parameters.

In particular, when $r=2$, that is, $x\in I$, we have
\begin{align}\label{2semi}
\bigg| \sup_{x \in I} \frac{d^{m}}{dx } P_t f (x) \bigg| \leq \bigg| \sup_{x \in I} \frac{d^m}{dx} f (x) \bigg| \exp{ \bigg[ -t\tilde{\lambda}_{m}\bigg]},
\end{align}
where $\tilde{\lambda}_{m}=\frac{m(m-1)}{4N}+m(u_{12}+u_{21})$.
\end{lem}

\prf Given any $f \in \mathcal{C}^{4}(I)$, we define function $F$ on $\RR_+ \times I$ as $F(t, x) := P_{t} f(x)$. According to \cite{Eth76}, $P_t$ constitutes a strongly continuous semigroup, $F$ is differentiable with respect to $t$. Equation (\ref{power}) implies that $F (t, \cdot) \in \mathcal{C}^{4}(I)$ for each $t \geq 0$. Furthermore,  using Chapman-Kolmogorov equation and (\ref{opde}), $F$ satisfies the Kolmogorov backward equation, i.e.,
$$
\frac{\partial}{\partial t} F (t, x) = \L F(t,x), \qqd F(0, x) = f(x).
$$
For each $k_1 \in \{1, 2, \ldots, r-1\}$, denote by $F_{k_1}^{(1)} (t, x) = \dps\frac{\partial}{\partial x_{k_1}} F(t, x)$ the partial derivative of $F(t, x)$ with respect to $x_{k_1}$. By the definition of $a_{ij} (x)$ and $b_i (x)$ given in (\ref{arr}) and (\ref{mul}), we have
$$
\frac{\p}{\p x_{k_1}} a_{ij} (x) =
\left\{
  \begin{array}{ll}
    1- 2 x_{k_1}, & \hbox{$i= j= k_1$,} \\
    \dz_{k_1j} - x_j, & \hbox{$i = k_1$, $j \neq k_1$,} \\
    - x_i, & \hbox{$i \neq k_1$, $j = k_1$,} \\
    0, & \hbox{$i \neq k_1$, $j \neq k_1$,}
  \end{array}
\right.
$$
and
$$
\frac{\p}{\p x_{k_1}} b_i(x) =
\left\{
  \begin{array}{ll}
    - (\sum_{j= 1}^{r} u_{k_1j}) - u_{rk_1}, & \hbox{$i = k_1$,} \\
    u_{k_1i} - u_{ri}, & \hbox{$i \neq k_1$.}
  \end{array}
\right.
$$
Hence, using the Assumption \ref{ass}, it holds that
\be\label{first}
\frac{\partial}{\partial t} F_{k_1}^{(1)} (t, x) = \L_{1} F_{k_1}^{(1)} (t, x) - \lt( \sum_{j=1}^{r} u_{k_1j} + u_{rk_1} \rt) F^{(1)}_{k} (t, x),
\de
where
$$
\L_{1} = \L + \frac{1}{4N} \sum_{i=1}^{r-1} \lt( \delta_{k_1i} - 2x_{i} \rt) \frac{\partial}{\partial x_{i}}.
$$
As per  \cite{Eth76}, it follows that the first derivative of  $F_{k_1}^{(1)}$ with respect to  $t \in \RR_+$ and the second derivative concerning $x \in I$ are continuous. Considering all $t \geq 0$ and $x \in I$, let's denote
$$
\qz^{(1)}_{k_1} (t, x) = F_{k_1}^{(1)} (t, x) \exp\lt[ \lt( \sum_{j=1}^{r} u_{k_1j} + u_{rk_1} \rt) t\rt].
$$
Then, equation (\ref{first}) implies that
$$
\frac{\partial}{\partial t} \qz^{(1)}_{k_1} (t, x) = \L_{1} \qz^{(1)}_{k_1} (t, x), \qqd \text{where $\qz_{k_1}^{(1)} (0, x) = \frac{\p}{\p x_{k_1}} f(x) =: f_{k_1}^{(1)} (x)$}.
$$
Therefore, using the weak maximum principle for parabolic equations (cf. \cite[Theorem 2.4]{L96} or \cite[Theorem 6.3.1]{Fri75}), we obtain that $\Big| \dps \sup_{(t, x) \in \RR_+ \times I} \psi^{(1)}_{k_1} (t, x) \Big| \leq \Big| \dps \sup_{x \in I} f_{k_1}^{(1)} \Big|$. Hence, it holds that
$$
\lt| \sup_{x \in I} \frac{\p}{\p x_{k_1}} (P_t f) (x) \rt| \leq \lt| \sup_{x \in I}  f_{k_1}^{(1)} (x) \rt| \exp{ \bigg[ -t \bigg( \sum_{j=1}^{r} u_{k_1j} + u_{rk_1} \bigg) \bigg]}, \qd \forall t \geq 0,
$$
which implies that (\ref{semi}) holds for $m=1$.

For each $k_1, k_2 \in \{1, 2, \ldots, r-1\}$, denote by $\dps\frac{\partial^{2}}{\partial x_{k_2} \partial x_{k_1}} F(t, x) = F^{(2)}_{k_1, k_2} (t,x)$. Since $F (t, \cdot) \in \mathcal{C}^{4}(I)$, we differentiate both sides of (\ref{first}) with respect to $x_{k_2}$ and then
\be\lb{second}
\frac{\partial}{\partial t} F^{(2)}_{k_1, k_2} (t, x) = \L_{2} F^{(2)}_{k_1, k_2} (t, x) - \lt[ \sum_{j=1}^r (u_{k_1j} + u_{k_2j}) + u_{rk_2} + u_{rk_1} + \frac{1}{2N} \rt] F^{(2)}_{k_1, k_2} (t, x),
\de
where
$$
\L_{2} = \L + \frac{1}{4N} \sum_{i=1}^{r-1} \lt( \dz_{k_1i} + \dz_{k_2i} - 4 x_i \rt) \frac{\p}{\p x_{i}}.
$$
Denote by $\psi^{(2)}_{k_1, k_2} (t, x) = F^{(2)}_{k_1, k_2} (t, x) \exp\lt[ \lt(\sum_{j=1}^{r} (u_{k_1j} + u_{k_2j}) + u_{rk_2} + u_{rk_1} + \dps\frac{1}{2N}\rt) t \rt]$, and then we have
$$
\frac{\partial}{\partial t} \qz^{(2)}_{k_1, k_2} (t, x) = \L_{2} \qz^{(2)}_{k_1, k_2} (t, x), \qqd \text{where $\qz_{k_1, k_2}^{(2)} (0, x) = \frac{\p^2}{\p x_{k_1} \p x_{k_2}} f(x) =: f_{k_1, k_2}^{(2)} (x)$}.
$$
By the same argument as $F^{(1)}_{k_1} (t, x)$, using the weak maximum principle, then we have
$$
\Big| \dps \sup_{(t, x) \in \RR_+ \times I} \psi^{(2)}_{k_1, k_2} (t, x) \Big| \leq \Big| \dps \sup_{x \in I} f_{k_1, k_2}^{(2)} (x) \Big|,
$$
which means that
\begin{align*}
&\lt| \sup_{x \in I} \frac{\p^2}{\p x_{k_2} \p x_{k_1}} P_t f (x) \rt| \\
\leq& \lt| \sup_{x \in I} \frac{\p^2}{\p x_{k_2} \p x_{k_1}} f (x) \rt| \exp{ \bigg[ -t \bigg( \sum_{j=1}^{r} (u_{k_1j} + u_{k_2j}) + u_{rk_2} + u_{rk_1} + \frac{1}{2N} \bigg) \bigg]}, \qd \forall t \geq 0.
\end{align*}
Hence, we obtain (\ref{semi}) in the case $m=2$.

Next, we consider the case of $i=3, 4$. Since $f \in \mathcal{C}^{4} (I)$, it holds that $P _{t} f \in \mathcal{C}^{4} (I)$ and $\dps\frac{\p^3}{\p x_{k_3} \p x_{k_2} \p x_{k_1}} P_t f \in \mathcal{C}^{(1)}(I)$. We will employ a standard polishing technique aided by the normal distribution (see, e.g., \cite[Section 7]{FSX19}). Denote
\begin{align*}
(P_{t} f)^{\epsilon} (x) = \int_{\mathbb{R}^{r-1}} a_{\epsilon} (y) (P_{t} f) (x-y) \d y \qd \text{and} \qd f^{\epsilon} (x) = \int_{\mathbb{R}^{r-1}} a_{\epsilon} (y) f (x-y) \d y,
\end{align*}
where $\epsilon>0$ and $a_{\epsilon}$ is the density function of the normal distribution $N (0, \epsilon^{2} E_{r-1})$ ($E_{r-1}$ is the $(r-1)$-dimensional identity matrix). It can be verified that $( P_{t} f)^{\epsilon}$ is smooth and
$$
\lim_{\epsilon \to 0} (P_{t} f)^{\epsilon} (x) = P_{t} f (x) \qd \text{and} \qd \Big| \sup_{x \in I} (P_{t} f)^{\epsilon} (x)\Big| \leq \Big| \sup_{x \in I} P_t f (x) \Big|, \qd \forall f \in \mathcal{C}^{4} (I).
$$
Denote
\begin{align*}
&F^{\epsilon} (t, x) = (P_{t} f)^{\epsilon} (x), \quad F^{\epsilon, (1)}_{k_1} (t, x) = \dps\frac{\p}{\p x_{k_1}} F^{\epsilon} (t, x),\\
&F^{\epsilon, (2)}_{k_1, k_2} (t, x) = \dps\frac{\p^2}{\p x_{k_2} \p x_{k_1}} F^{\epsilon} (t, x).
\end{align*}
By these definitions, it can be shown that (\ref{first}) and (\ref{second}) hold for $F^{\epsilon, (1)}_{k_1} (t, x)$ and $F^{\epsilon, (2)}_{k_1, k_2}$.

For each $k_1, k_2, k_3 \in \{1, 2, \ldots, r-1\}$, we take the derivative of $F^{\epsilon, (2)}_{k_1, k_2}$ in (\ref{second}) with respect to $x_{k_3}$ and write $F^{\epsilon, (3)}_{k_1, k_2, k_3} := \dps\frac{\p^3}{\p x_{k_3} \p x_{k_2} \p x_{k_1}} F^{\epsilon} (t, x)$. Hence,
\begin{align*}
\frac{\p}{\p t} F_{k_1, k_2, k_3}^{\epsilon, (3)} (t, x) = \L_{3} F_{k_1, k_2, k_3}^{\epsilon, (3)} (t, x) - \lt[ \sum_{j=1}^{r} \sum_{n=1}^3 u_{k_n j} + \sum_{n=1}^3 u_{r k_n} + \frac{3}{2N} \rt] F_{k_1, k_2, k_3}^{\epsilon, (3)} (t, x),
\end{align*}
where
$$
\L_3 = \L + \frac{1}{4N} \sum_{i=1}^{r-1} \lt( \dz_{k_1 i} + \dz_{k_2 i} + \dz_{k_3 i} - 6 x_i \rt) \frac{\p}{\p x_i}.
$$
Denote by
$$
\qz_{k_1, k_2, k_3}^{\epsilon, (3)} (t, x) = F_{k_1, k_2, k_3}^{\epsilon, (3)} (t, x) \exp \bigg[ t \bigg( \sum_{j=1}^{r} \sum_{n=1}^3 u_{k_n j} + \sum_{n=1}^3 u_{r k_n} + \frac{3}{2N} \bigg) \bigg],
$$
and then (\ref{third}) implies that
\begin{align*}
\frac{\p}{\p t} \qz_{k_1, k_2, k_3}^{\epsilon, (3)} (t, x) = \L_{3} \psi_{k_1, k_2, k_3}^{\epsilon, (3)} (t, x),
\end{align*}
where $\qz_{k_1, k_2, k_3}^{\epsilon, (3)} (0, x) = \frac{\p^3}{\p x_{k_3} \p x_{k_2} \p x_{k_1}} f^{\epsilon} (x) =: f_{k_1, k_2, k_3}^{\epsilon, (3)} (x)$.
Again, the weak maximum principle ensures that $\Big| \dps\sup_{(t, x) \in \RR_+ \times I} \psi_{k_1, k_2, k_3}^{\epsilon, (3)} (t, x) \Big| \leq \Big| \sup_{x \in I} f_{k_1, k_2, k_3}^{\epsilon, (3)}(x) \Big|$, which means that
\begin{align*}
&\lt| \sup_{x \in I} \frac{\p^3}{\p x_{k_3} \p x_{k_2} \p x_{k_1}} (P_t f)^{\epsilon} (x) \rt|\\
\leq& \lt| \sup_{x \in I} f_{k_1, k_2, k_3}^{\epsilon, (3)} (x) \rt| \exp \bigg[ - t \bigg( \sum_{j=1}^{r} \sum_{n=1}^3 u_{k_n j} + \sum_{n=1}^3 u_{r k_n} + \frac{3}{2N} \bigg) \bigg].
\end{align*}
Therefore, passing the limit as $\epsilon \to 0$, we obtain (\ref{semi}) of $m=3$.

In exactly the same way, for each $k_1, k_2, k_3, k_4 \in \{1, 2, \ldots, r-1\}$, we have
\begin{align*}
\frac{\p}{\p t} F_{k_1, \ldots, k_4}^{\epsilon, (4)} (t, x)
=& \L_{4} F_{k_1, \ldots, k_4}^{\epsilon, (4)} (t, x) - \lt[ \sum_{j=1}^{r} \sum_{n=1}^4 u_{k_n j} + \sum_{n=1}^4 u_{r k_n} + \frac{3}{N} \rt] F_{k_1, \ldots, k_4}^{\epsilon, (4)} (t, x),
\end{align*}
where
$$
\L_4 = \L + \frac{1}{4N} \sum_{i=1}^{r-1} \lt( \dz_{k_1 i} + \dz_{k_2 i} + \dz_{k_3 i} + \dz_{k_4 i} - 8 x_i \rt) \frac{\p}{\p x_i}.
$$
Moreover, it holds that
\begin{align*}
&\lt| \sup_{x \in I} \frac{\p^4}{\p x_{k_1} \cdots \p x_{k_4}} (P_t f)^{\epsilon} (x) \rt|\\
\leq& \lt| \sup_{x \in I} f_{k_1, \ldots, k_4}^{\epsilon, (4)} (x) \rt| \exp \bigg[ - t \bigg( \sum_{j=1}^{r} \sum_{n=1}^4 u_{k_n j} + \sum_{n=1}^4 u_{r k_n} + \frac{3}{N} \bigg) \bigg].
\end{align*}
Finally, (\ref{semi}) holds by passing the limit as $\epsilon \to 0$.

In addition, when $r=2$, (\ref{2semi}) can be proved in the same way.
\deprf

Regarding the diffusion process $X (t)$, Lemma \ref{regular} indicates  that the higher-order derivative of the semi-group $P_t f$ can be bounded by the derivatives of $f$. As $Y (t)$ operates as a discrete time Markov chain, in order to characterize the effect of approximating $Y(t)$ with $X(t)$, we need to give the error estimation of these two processes with $t=1$.

\begin{lem}\label{compare}
Suppose that Assumption \ref{ass} holds and $r\geq2$. Let $(X (t))_{t \in \RR_+}$ be the diffusion process defined in (\ref{WF-CH}) and $(Y (n))_{n \in \NN}$ be the Markov chain defined in (\ref{WF-D-r}). Then, for any $f \in \mathcal{C}_{b}^{4} (I)$ and $y \in I_{2N} \subset I$,
\begin{itemize}
\item[(i)] when $r\geq3$, it holds that
$$
\lt| \EE \lt[f (X (1)) \big| X (0) = y \rt] - \EE \lt[ f (Y(1)) \big| Y(0) = y \rt] \rt| \leq \sum_{n=1}^4 C_n \lt( \sup_{y \in I} \lt\| \nabla^n f (y) \rt\|_2 \rt),
$$
where $C_n$ ($n=1, \ldots, 4$) are constants defined in Theorem \ref{main1};
\item[(ii)] when $r=2$, it holds that
$$
\lt| \EE \lt[f (X (1)) \big| X (0) = y \rt] - \EE \lt[ f (Y(1)) \big| Y(0) = y \rt] \rt| \leq \sum_{n=1}^4 \tilde{C}_n \lt( \sup_{y \in I} \lt\| \nabla^n f (y) \rt\|_2 \rt),
$$
where $\tilde{C}_n$ ($n=1, \ldots, 4$) are constants defined in Corollary \ref{cor}.
\end{itemize}
\end{lem}

\prf (i) When $r\geq3$, we can prove the conclusion in three steps.
\newline
\underline{Step $1$.} We consider the diffusion process $X (t)$ defined in (\ref{WF-CH}) first. Using (\ref{power}), for all $f \in \mathcal{C}_{b}^{4} (\RR^{r-1})$ and $y \in I$, we obtain the Taylor expansion of $P_{t} f$ as
\begin{align}\lb{Taylor}
\EE \lt[ f (X(1)) \big| X(0) = y \rt] = P_1 f(y) = f(y) + \L f (y) + \frac{1}{2} P_{\kz} \lt( \L^2 f \rt) (y),
\end{align}
where $\L^2 f := \L (\L f)$ and $\kz \in (0, 1)$. According to (\ref{oper2}), we have
\begin{align}\lb{Taylor2}
\L^2 f(y) &= \frac{1}{4N} \langle A(y), \nabla^2 (\L f) (y) \rangle_{\rm HS} + \langle b(y), \nabla (\L f) (y) \rangle \notag \\
& \leq \frac{1}{4N} \lt( \sup_{y \in I} \| A(y) \|_{\rm HS} \rt) \cdot \lt\| \nabla^2 (\L f) (y) \rt\|_{\rm HS} + \lt( \sup_{y \in I} \| b(y) \| \rt) \cdot \lt\| \nabla (\L f) (y) \rt\|.
\end{align}
Referring to the expression for $A(y)$ provided in (\ref{arr}), we derive
$$
\sup_{y \in I} \| A(y) \|_{\rm HS} = \sup_{y \in I} \lt[ \sum_{i, j =1}^{r-1} \lt( y_i (\dz_{ij} - y_j) \rt)^2 \rt]^{1/2} = \frac{\sqrt{r-2}}{r-1},
$$
where the supremum is  achieved when  $y_i = (r-1)^{-1}$ for all $i \in \{1, \ldots, r-1\}$. Under Assumption \ref{ass}, it holds that $b_i (y) = - y_i \sum_{j=1}^r u_{ij} + u_{ri} - u_{ri} y_i$. Hence, we have
\begin{align*}
\sup_{y \in I} \| b (y) \| =& \sup_{y \in I} \lt[ \sum_{i=1}^{r-1} \lt( y_i \sum_{j=1}^r u_{ij} - u_{ri} + u_{ri} y_i \rt)^2 \rt]^{1/2}
 = \lt[ \lt( \sum_{j=1}^r u_{i^* j} \rt)^2 + \sum_{i \neq i^*} u_{ri}^2 \rt]^{1/2},
\end{align*}
where $i^* = \arg \max \lt( \sum_{j=1}^r u_{ij} + u_{ri}\rt)$ and the supremum is attained when $y_i = \dz_{i^*i} $ for all $i \in \{1, \ldots, r-1\}$. We denote by $b^*$ the constant of $\sup_{y \in I} \| b (y) \|$.

Subsequently, we proceed to compute  $\nabla^2 (\L f) (y)$ and $\nabla (\L f) (y)$. Let's start by considering  $\nabla (\L f) (y)$. Using (\ref{first}) in the proof of Lemma \ref{regular}, for each $k \in \{1, 2, \ldots, r-1 \}$, it holds that
\begin{align*}
\frac{\p}{\p y_k} (\L f) (y) & = \frac{1}{4N} \bigg\langle A(y), \nabla^2 \Big( \frac{\p}{\p y_k} f(y) \Big) \bigg\rangle_{\rm HS} + \bigg\langle b(y), \nabla \Big( \frac{\p}{\p y_k} f(y) \Big) \bigg\rangle \\
& \qd + \frac{1}{4N} \bigg\langle \dz_{k \bullet} - 2y, \nabla \Big( \frac{\p}{\p y_k} f(y) \Big) \bigg\rangle - \bigg( \sum_{j=1}^r u_{kj} + u_{rk} \bigg) \Big( \frac{\p}{\p y_k} f(y) \Big),
\end{align*}
where $\dz_{k \bullet} = (\dz_{k i})_{i \in \{1, \ldots r-1\}} \in \RR^{r-1}$. For each $\cz \in \RR^{r-1}$ satisfying $\| \cz\| \leq 1$, we have
$$
\langle \nabla (\L f) (y), \cz \rangle =: K_1 + K_2 + K_3 + K_4.
$$
The expressions and estimations for $K_m$ ($m=1, \ldots, 4$) are presented below:
\begin{align*}
K_1 & = \frac{1}{4N} \sum_{k=1}^{r-1} \cz_k \bigg( \sum_{i, j=1}^{r-1} a_{ij} (y) \frac{\p^2}{\p y_i \p y_j} \Big( \frac{\p}{\p y_k} f(y) \Big) \bigg)\\
 &\leq \frac{1}{4N} \lt( \sum_{i, j, k=1}^{r-1} \lt( \cz_k a_{ij} (y) \rt)^2 \rt)^{1/2} \lt( \sum_{i, j, k=1}^{r-1} \lt( \frac{\p^3}{\p y_k \p y_i \p y_j} f(y) \rt)^2 \rt)^{1/2} \\
& = \frac{1}{4N} \| \cz \| \cdot \| A (y)\|_{\rm HS} \lt\| \nabla^3 f(y) \rt\|_2, \\
K_2 & = \sum_{k=1}^{r-1} \cz_k \bigg( \sum_{i=1}^{r-1} b_i (y) \frac{\p^2}{\p y_k \p y_i} f (y) \bigg)
 = \lt\langle b(y) \cz^\top, \nabla^2 f (x) \rt\rangle_{\rm HS} \leq \| \cz\| \cdot \| b (y) \| \lt\| \nabla^2 f (y) \rt\|_{2}, \\
K_3 & = \frac{1}{4N} \sum_{k=1}^{r-1} \cz_k \bigg( \sum_{i=1}^{r-1} (\dz_{ki} - 2 y_i) \frac{\p^2}{\p y_i \p y_k} f (y) \bigg) = \frac{1}{4N} \lt\langle {\rm diag}(\cz) - 2 y \cz^\top, \nabla^2 f (y) \rt\rangle_{\rm HS} \\
& \leq \frac{1}{4N} \lt\| \nabla^2 f (y) \rt\|_{2} \cdot \| {\rm diag} (\cz) - 2 y \cz^\top \|_{\rm HS}\\
 &\leq \frac{1}{4N} \lt\| \nabla^2 f (y) \rt\|_{2} \sqrt{\| \cz \|^2 + 4 \| y\|^2 \| \cz\|^2  - 4 \langle y, \cz^2 \rangle }, \\
K_4 & = \sum_{k=1}^{r-1} (- \cz_k) \bigg( \sum_{j=1}^r u_{kj} + u_{rk} \bigg) \frac{\p}{\p y_k} f (y) \leq u^* \lt\| \nabla f (y) \rt\|_2,
\end{align*}
where $u^* := \dps\max_{1 \leq k \leq r-1} \bigg( \sum_{j=1}^r u_{kj} + u_{rk} \bigg)$. Combining the estimations of $K_m$ ($m=1, \ldots, 4$), we obtain
\begin{align}\lb{nabla1}
& \lt\| \nabla (\L f) (y) \rt\| = \sup_{\| \cz\| \leq 1} \langle \nabla (\L f) (y), \cz \rangle \notag \\
& \qd \leq \frac{\sqrt{r-2}}{4N (r-1)} \lt\| \nabla^3 f (y) \rt\|_2 + \lt( b^* + \frac{1}{4N} \sqrt{1+ 4 \| y\|^2} \rt) \lt\| \nabla^2 f (y) \rt\|_2 + u^* \lt\| \nabla f (y) \rt\|_2.
\end{align}

Now, we consider $\nabla^2 (\L f) (y)$. By (\ref{relation}), we have
\begin{align*}
\lt\| \nabla^2 (\L f) (y) \rt\|_{\rm HS} \leq& \sqrt{r-1} \lt\| \nabla^2 (\L f) (y) \rt\|_{\rm op}\\
 =& \sqrt{r-1} \sup_{\| \az \| \leq 1 \atop \| \bz \| \leq 1} \lt\{ \sum_{k, l =1}^{r-1} \bigg( \frac{\p^2}{\p y_k \p y_l} (\L f) (y) \bigg) \az_k \bz_l \rt\}.
\end{align*}
According to (\ref{second}) in Lemma \ref{regular}, it holds that
\begin{align*}
\frac{\p^2}{\p y_k \p y_l} (\L f) (y) & = \frac{1}{4N} \bigg\langle A(y), \nabla^2 \Big( \frac{\p^2}{\p y_k \p y_l} f (y) \Big) \bigg\rangle_{\rm HS} + \bigg\langle b(y), \nabla \Big( \frac{\p^2}{\p y_k \p y_l} f (y) \Big) \bigg\rangle \\
& \qd + \frac{1}{4N} \bigg\langle \dz_{k \bullet} + \dz_{l \bullet} - 4 y, \nabla \Big( \frac{\p^2}{\p y_k \p y_l} f (y) \Big) \bigg\rangle\\
 &\qd- \Big[ \sum_{j=1}^r (u_{kj} + u_{lj}) + u_{rk} + u_{rl} + \frac{1}{2N} \Big] \Big( \frac{\p^2}{\p y_k \p y_l} f (y) \Big).
\end{align*}
Hence, for each $\az, \bz \in \RR^{r-1}$ satisfying $\| \az \| \leq 1$ and $\| \bz \| \leq 1$, we have
$$
\sum_{k, l =1}^{r-1} \bigg( \frac{\p^2}{\p y_k \p y_l} (\L f) (y) \bigg) \az_k \bz_l =: M_1 + M_2 + M_3 + M_4,
$$
where the expressions and estimations for $M_m$ ($m=1, \ldots, 4$) are
\begin{align*}
M_1 &= \frac{1}{4N} \sum_{k, l=1}^{r-1} \az_k \bz_l \sum_{i, j=1}^{r-1} a_{ij} (y) \frac{\p^2}{\p y_j \p y_i} \Big( \frac{\p^2}{\p y_k \p y_l} f (y)\Big)\\
 &\leq \frac{1}{4N} \| \az \| \cdot \| \bz\| \cdot \| A(y)\|_{\rm HS} \lt\| \nabla^4 f (y) \rt\|_2, \\
M_2 &= \sum_{k, l =1}^{r-1} \az_k \bz_l \sum_{i=1}^{r-1} b_i (y) \frac{\p}{\p y_i} \Big( \frac{\p^2}{\p y_k \p y_l} f(y) \Big) \leq \| \az \| \cdot \| \bz\| \cdot \| b(y) \| \cdot \lt\| \nabla^3 f (y) \rt\|_2 \\
M_3 &= \frac{1}{4N} \sum_{k, l=1}^{r-1} \az_k \bz_l \sum_{i=1}^{r-1} (\dz_{ki} + \dz_{li} - 4 y_i) \frac{\p}{\p y_i} \Big(\frac{\p^2}{\p y_k \p y_l} f(y) \Big) \\
& \leq \frac{1}{4N} \lt\| \nabla^3 f (y)\rt\|_2 \lt[ \sum_{i, k, l =1}^{r-1} \lt( \az_k \bz_l \dz_{ki} + \az_k \bz_l \dz_{li} - 4 y_i \az_k \bz_l \rt)^2 \rt]^{1/2} \\
& \leq \frac{1}{4N} \lt\| \nabla^3 f (y)\rt\|_2\\
 &\qd\cdot\sqrt{2 (\| \az \| \cdot \| \bz \|)^2 + 16 (\| \az \| \cdot \| \bz \|)^2 \| y \|^2 - 8 \| \az \|^2 \lt\| \bz \sqrt{y} \rt\|^2 - 8 \| \bz \|^2 \lt\| \az \sqrt{y} \rt\|^2}, \\
M_4 &= \sum_{k, l=1}^{r-1} (- \az_k \bz_l) \bigg( \sum_{j=1}^{r} (u_{kj} + u_{lj}) + u_{rk} + u_{rl} + \frac{1}{2N} \bigg) \Big( \frac{\p^2}{\p y_k \p y_l} f(y) \Big) \\
& \leq \bigg( 2 u^* + \frac{1}{2N} \bigg) \, \| \az \| \cdot \| \bz\| \cdot \lt\| \nabla^2 f (y)\rt\|_2.
\end{align*}
Hence, we obtain the estimation of $\lt\| \nabla^2 (\L f) (y) \rt\|_{\rm HS}$ as
\begin{align}\lb{nabla2}
& \lt\| \nabla^2 (\L f) (y) \rt\|_{\rm HS} \leq \sqrt{r-1} \sup_{\| \az \| \leq 1 \atop \| \bz \| \leq 1} \lt\{ \sum_{k, l =1}^{r-1} \bigg( \frac{\p^2}{\p y_k \p y_l} (\L f) (y) \bigg) \az_k \bz_l \rt\} \notag \\
& \qd \leq \sqrt{r-1} \lt[ \frac{\sqrt{r-2}}{4N (r-1)} \lt\| \nabla^4 f (y) \rt\|_2 + \lt( b^* + \frac{1}{4N} \sqrt{2+ 16 \| y\|^2} \rt) \lt\| \nabla^3 f (y)\rt\|_2\right.\nonumber\\
&\left.\qd\qd + \lt(2 u^* + \frac{1}{2N} \rt) \lt\| \nabla^2 f(y) \rt\|_2 \rt].
\end{align}

By combining  (\ref{Taylor}), (\ref{Taylor2}), (\ref{nabla1}) and (\ref{nabla2}), we derive that
\begin{align}\lb{Taylor3}
&\EE \lt[ f (X_1) \big| X(0) = y \rt]\nonumber\\
  =& f(y) + \frac{1}{4N} \langle A(y), \nabla^2 f (y) \rangle_{\rm HS} + \langle b(y), \nabla f (y) \rangle + \frac{1}{2} P_{\kz} \lt( \L^2 f \rt) (y),
\end{align}
where $P_{\kz} \lt( \L^2 f \rt) (y)$ satisfies the following estimation
\begin{align*}
&P_{\kz} \lt( \L^2 f \rt) (y) \leq \lt| \L^2 f (y) \rt|\\
 &\qd \leq \frac{r-2}{16 N^2 (r-1)^{3/2}} \lt\| \nabla^4 f (y) \rt\|_2 + \frac{\sqrt{r-2}}{4N \sqrt{r-1}} \lt( b^* + \frac{b^*}{\sqrt{r-1}} + \frac{3 \sqrt{2}}{4N} \rt) \lt\| \nabla^3 f (y) \rt\|_2 \\
& \qd + \lt[ \frac{\sqrt{r-2}}{4N \sqrt{r-1}} \bigg( 2 u^* + \frac{1}{2N} \bigg) + b^* \bigg( b^* + \frac{\sqrt{5}}{4N} \bigg) \rt] \lt\| \nabla^2 f (y) \rt\|_2 + b^* u^* \lt\| \nabla f (y) \rt\|_2.
\end{align*}

\underline{Step $2$.} Next, we focus on the Markov chain $Y (n)$ defined in (\ref{WF-D-r}). Using the Taylor expansion again, for each $f \in \mathcal{C}_{b}^{4} (\RR^{r-1})$ and $y \in I_{2N}$, we have
\begin{align}\lb{Taylor4}
\EE_y \lt[ f(Y(1)) \rt] & := \EE \lt[ f(Y(1)) \big| Y(0) = y \rt] \notag \\
& = f (y) + \lt\langle \nabla f (y), \EE_y \lt[  Y(1) - y \rt] \rt\rangle \nonumber\\
 &\qd+ \frac{1}{2} \lt\langle \nabla^2 f(y), \EE \lt[ \lt(Y(1) - y \rt) \lt(Y(1) - y \rt)^\top \rt] \rt\rangle_{\rm HS} + \mathcal{R}_3,
\end{align}
where $\mathcal{R}_3$ is the remainder term of the Taylor expansion, which is given by
\begin{align*}
\mathcal{R}_3 := \frac{1}{2} \int_0^1 (1-t)^2 \sum_{i, j, k=1}^{r-1} \EE_y &\lt[ \Big(Y_i (1) - y_i \Big) \Big(Y_j (1) - y_j \Big)\right.\\
 &\left.\qd\cdot\Big(Y_k (1) - y_k \Big) \frac{\p^3}{\p y_k \p y_j \p y_i} f(y+ t (Y(1) - y)) \rt]\d t.
\end{align*}
Notice that when $Y(0) = (y_1, \ldots, y_{r-1})$, we have
$$
Y(1) \sim (2N)^{-1} { \rm multinomial } \lt( 2N, \ \big( y_1^{\#}, \ldots, y_{r-1}^{\#} \big) \rt),
$$
then for each $i \in \{1, 2, \ldots, r-1 \}$, we have
$$
\EE \lt[ Y_i (1) \big| Y (0) = y \rt] = y_i^{\#}, \qd \EE \lt[ Y^2_i (1) \big| Y (0) = y \rt] = \Big( 1 - \frac{1}{2N}\Big) \lt( y_i^{\#} \rt)^2 + \frac{1}{2N} y_i^{\#},
$$
\begin{align}\label{third}
&\mathbb{E}_{y}\left[Y_i(1) - y_i \right]^3\nonumber\\
=&\frac{y_{i}^{\#}}{(2N)^{2}}(1-2y_{i}^{\#})(1-y_{i}^{\#})+\frac{3y_{i}^{\#}}{2N}(1-y_{i}^{\#})(y_{i}^{\#}-y_{i})+(y_{i}^{\#}-y_{i})^{3}.
\end{align}
Furthermore, using (\ref{muta-r}) and (\ref{mul}), we obtain
\begin{align*}
& \EE \lt[ Y(1) - y \big| Y (0) = y \rt] = \lt( y^{\#}_i - y_i \rt)_{i \in \{1, \ldots, r-1 \}} = b (y), \\
& \EE \lt[ \lt(Y(1) - y \rt) \lt( Y(1) - y\rt)^\top \big| Y (0) = y \rt] = \lt( \hat{a}_{ij} (y)\rt)_{i, j \in \{1, \ldots, r-1 \}} =: \hat{A} (y),
\end{align*}
here $y^{\#}_i$  represents the adjusted proportion of the allele $A_i$ defined in (\ref{muta-r}), and the matrix $\hat{A} (y)$ is specified as
$$
\hat{a}_{ij} =
\left\{
  \begin{array}{ll}
    \lt( y_i^{\#} - y_i \rt) \lt( y_j^{\#} - y_j\rt) - \dps\frac{1}{2N} y_i^{\#} y_j^{\#}, & \hbox{$i \neq j$;} \\
    (y_i^{\#} - y_i)^2 + \dps\frac{1}{2N} y_i^{\#} (1 - y_i^{\#}), & \hbox{$i = j$.}
  \end{array}
\right.
$$
Hence, by (\ref{Taylor4}), we deduce
\be\lb{Taylor5}
\EE \lt[ f(Y(1)) \big| Y(0) = y \rt] = f (y) + \lt\langle \nabla f (y), b(y) \rt\rangle + \frac{1}{2} \lt\langle \nabla^2 f(y), \hat{A} (y) \rt\rangle_{\rm HS} + \mathcal{R}_3,
\de
where $\mathcal{R}_3$ is estimated as follows
$$
|\mathcal{R}_3| \leq \frac{1}{6} \lt( \sup_{y \in I_{2N}} \lt\| \nabla^3 f (y) \rt\|_2 \rt) \lt( \sup_{y \in I_{2N}} \EE_y \lt[ \| Y(1) - y \|^3 \rt] \rt).
$$
Using  the H$\ddot{o}$lder inequality and (\ref{third}), we have
\begin{align*}
\EE_y \lt[ \| Y(1) - y \|^3 \rt] & \leq \lt( \sqrt{r-1} \rt)  \EE_y \lt[ \sum_{i=1}^{r-1} ( Y_i(1) - y_i )^3 \rt] \\ &\leq(r-1)^{\frac{3}{2}}\left[\frac{1}{32N^{2}}+\frac{3u^{*}}{8N}+(u^{*})^{3}\right].
\end{align*}

\underline{Step $3$.} Referring to equations (\ref{Taylor3}) and (\ref{Taylor5}), we obtain that
\begin{align*}
\lt| \EE_y \lt(f (X_1) \rt) - \EE_y \lt( f (Y(1)) \rt) \rt|
=&\left| \langle \nabla^2 f (y), \frac{1}{4N} A (y) - \frac{1}{2} \hat{A} (y) \rangle_{\rm HS} + \frac{1}{2} P_{\kz} (\L^2 f) (y) - \mathcal{R}_3 \right|\\
\leq& \lt\| \frac{1}{4N} A (y) - \frac{1}{2} \hat{A} (y) \rt\|_{\rm HS} \lt\| \nabla^2 f (y) \rt\|_2 + \frac{1}{2} \lt| \L^2 f (y)\rt| + |\mathcal{R}_3|.
\end{align*}
Using the fact that $y^{\#}_i - y_i = b_i (y)$, it holds that
\begin{align*}
&\lt\| \frac{1}{4N} A (y) - \frac{1}{2} \hat{A} (y) \rt\|_{\rm HS}^2\\
=& \sum_{i,j=1}^{r-1}\left\{\frac{1}{4N}\left[y_{i}(\delta_{ij}-y_{j})-y_{i}^{\#}(\delta_{ij}-y_{j}^{\#})\right]-\frac{1}{2}(y_{i}^{\#}-y_{i})(y_{j}^{\#}-y_{j})\right\}^{2}\\
\leq&\sum_{i,j=1}^{r-1}\frac{1}{8N^{2}}\left[y_{i}(\delta_{ij}-y_{j})-y_{i}^{\#}(\delta_{ij}-y_{j}^{\#})\right]^{2}+\frac{1}{2}(y_{i}^{\#}-y_{i})^{2}(y_{j}^{\#}-y_{j})^{2}\\
\leq&\frac{(r-1)^{2}(u^{*})^{2}}{2}\left[\frac{1}{N^{2}}+(u^{*})^{2}\right].
\end{align*}
Therefore, by sorting out the estimations of $\lt| \L^2 f (y)\rt|$ and $\mathcal{R}_3$, we conclude that
\begin{align*}
&\lt| \EE_y \lt(f (X(1)) \rt) - \EE_y \lt( f (Y(1)) \rt) \rt|\\
\leq& C_4 \lt\| \nabla^4 f (y) \rt\|_2 + C_3 \lt\| \nabla^3 f (y) \rt\|_2 + C_2 \lt\| \nabla^2 f (y) \rt\|_2 + C_1 \lt\| \nabla f (y) \rt\|_2,
\end{align*}
where $C_m$ ($m=1, \ldots, 4$) are defined in Theorem \ref{main1}, and the claim (i) is proved.

(ii) When $r=2$, notice that for any $y\in I_{2N}$,
\begin{align*}
A(y)=y(1-y), \quad b(y)=-u_{12}x+u_{21}(1-x),
\end{align*}
we have
\begin{align*}
\sup_{y\in I}|A(y)|=\frac{1}{4}, \quad \sup_{y\in I}|b(y)|=\max\{u_{12},u_{21}\}.
\end{align*}
Then, by the same argument as the proof of \eqref{Taylor3}, we obtain
\begin{align*}
&\EE \lt[ f (X(1)) \big| X(0) = y \rt] \nonumber\\
=& f(y) + \frac{1}{4N}y(1-y)f^{(2)}(y)+[-u_{12}y+u_{21}(1-y)]f^{(1)}(y) + \frac{1}{2} P_{\kz} \lt( \L^2 f \rt) (y),
\end{align*}
where $P_{\kz} \lt( \L^2 f \rt) (y)$ satisfies the following estimation
\begin{align*}
&P_{\kz} \lt( \L^2 f \rt) (y) \leq \lt| \L^2 f (y) \rt| \\
& \leq \frac{1}{256N^2}\left|f^{(4)}(y)\right| + \frac{1}{8N} \lt(\frac{1}{4N}+\max\{u_{12},u_{21}\}\rt)\left|f^{(3)}(y)\right|\\
& \qd + \lt(\frac{1}{32N^{2}}+\frac{u_{12}+u_{21}}{8N}+\max\{u_{12}^{2},u_{21}^{2}\}+\frac{\max\{u_{12},u_{21}\}}{4N}\rt) \lt|f^{(2)}(y)\rt|\\ &\qd+\max\{u_{12},u_{21}\}(u_{12}+u_{21})\lt|f^{(1)}(y)\rt|.
\end{align*}
Moreover, employing the same logic as used in proving \eqref{Taylor5}, we arrive at
\begin{align*}
\EE \lt[ f(Y(1)) \big| Y(0) = y \rt] =& f (y) + [-u_{12}y+u_{21}(1-y)]f^{(1)}(y)\nonumber\\
 &+ \frac{1}{2}\left[(y^{\#}-y)^{2}+\frac{1}{2N}y^{\#}(1-y^{\#})\right]f^{(2)}(y)+ \mathcal{R}_3,
\end{align*}
where $y^{\#}=(1-u_{12})y+u_{21}(1-y)$ and $\mathcal{R}_3$ has the following estimation
\begin{align*}
|\mathcal{R}_3|
\leq& \frac{1}{6} \lt( \sup_{y \in I_{2N}} \lt|f^{(3)}(y) \rt| \rt) \lt( \sup_{y \in I_{2N}} \EE_y \lt[ \| Y(1) - y \|^3 \rt] \rt)\\
=& \frac{1}{6} \lt( \sup_{y \in I_{2N}} \lt|f^{(3)}(y) \rt| \rt)\lt( \sup_{y \in I_{2N}}\left[\frac{y^{\#}}{(2N)^{2}}(1-2y^{\#})(1-y^{\#})+[-u_{12}y^{\#}+u_{21}(1-y^{\#})]^{3}\right.\right.\\
&\left.\left.\qquad\qquad\qquad\qquad\qquad\qquad+\frac{3y^{\#}}{2N}(1-y^{\#})[-u_{12}y^{\#}+u_{21}(1-y^{\#})]\right] \rt)\\
\leq&\left(\frac{1}{192N^{2}}+\frac{\max\{u_{12},u_{21}\}}{16N}+\frac{\max\{u_{12}^{3},u_{21}^{3}\}}{6}\right)\lt( \sup_{y \in I_{2N}} \lt|f^{(3)}(y) \rt| \rt).
\end{align*}
Hence, we have
\begin{align*}
&\lt| \EE_y \lt(f (X(1)) \rt) - \EE_y \lt( f (Y(1)) \rt) \rt|\\
=&\left| \frac{1}{2}\left[\frac{1}{2N}y(1-y)-(y^{\#}-y)^{2}-\frac{1}{2N}y^{\#}(1-y^{\#})\right]f^{(2)}(y) + \frac{1}{2} P_{\kz} (\L^2 f) (y) - \mathcal{R}_3 \right|\\
\leq&\frac{1}{2}\lt| \frac{1}{2N}y(1-y)-(y^{\#}-y)^{2}-\frac{1}{2N}y^{\#}(1-y^{\#}) \rt| \lt|f^{(2)}(y) \rt| + \frac{1}{2} \lt| \L^2 f (y)\rt| + |\mathcal{R}_3|.
\end{align*}
Using the fact that $y^{\#}-y=-u_{12}y+u_{21}(1-y)$, it holds that
\begin{align*}
\lt| \frac{1}{2N}y(1-y)-(y^{\#}-y)^{2}-\frac{1}{2N}y^{\#}(1-y^{\#}) \rt|\leq\frac{\max\{u_{12},u_{21}\}}{N}+\max\{u_{12}^{2},u_{21}^{2}\}.
\end{align*}
Therefore, upon organizing the estimates of $\lt| \L^2 f (y)\rt|$ and $\mathcal{R}_3$, we derive
\begin{align*}
\lt| \EE_y \lt(f (X(1)) \rt) - \EE_y \lt( f (Y(1)) \rt) \rt|
\leq& \tilde{C}_4 \lt|f^{(4)}(y)\rt| + \tilde{C}_3 \lt|f^{(3)}(y)\rt| + \tilde{C}_2 \lt|f^{(2)}(y)\rt| +\tilde{C}_1 \lt|f^{(1)}(y)\rt|,
\end{align*}
where $\tilde{C}_m$ ($m=1, \ldots, 4$) are defined in Corollary \ref{cor}, and the claim (ii) holds.
\deprf

\section{Proofs of Theorem \ref{main1} and Corollary \ref{cor}}\label{proof}

In this section, we first give the proof of the main theorem.

\subsection{Proof of Theorem \ref{main1}.} For clarity in the proof of this theorem, we'll use $X_t$ and $Y_n$ instead of denoting them as $X (t)$ and $Y (n)$. Given any $z \in I_{2N} \subset I$ and $k \in \mathbb{Z}^{+}$, let $\lt( \hat{X}_{t} (k, z) \rt)_{t \geq k}$ be the solution to the stochastic differential equation (\ref{WF-CH}) satisfying $\hat{X}_k (k, z) = z$. According to the Markov property, the semi-group of $X_t$ is consistent with the one of $\hat{X}_{t} (k, z)$. By this definition, the processes $X_t$ and $\hat{X}_{t}$ satisfy the following two relationships
\be\lb{relat1}
X_t \overset{\d}{=} \hat{X}_{t} (0, X_0) \overset{\d}{=} \hat{X}_{t} (k, X_{k}) \qd \text{and} \qd \hat{X}_t (k, z) \overset{\d}{=} \hat{X}_{t-k} (0, z), \qd \forall t \geq k, \  z \in I_{2N}.
\de
We assume that $X_0 = Y_0 = x$ for some $x \in I_{2N}$, then we have $\hat{X}_1 (0, X_0) \overset{\d}{=} \hat{X}_1 (0, Y_0)$. Hence, for $n \geq 1$, it holds that
$$
\EE_x \lt[ f (X_{n}) \rt] = \EE_x \lt[ f \Big( \hat{X}_{n} \big(1, \hat{X}_{1} (0, Y_0) \big) \Big) \rt] - \EE_x \lt[ f \big( \hat{X}_{n} (1, Y_{1}) \big) \rt] + \EE_x \lt[ f \big( \hat{X}_{n} (1, Y_{1}) \big) \rt].
$$
For $n \geq 2$, equation (\ref{relat1}) implies that $\hat{X}_{n} (1, y_1) \overset{\d}{=} \hat{X}_{n} \big( 2, \hat{X}_{2} (1, y_1) \big)$ for each $y_1 \in I_{2N}$. Now, for each  $n \geq 0$ and $x, y \in I_{2N}$,  let  $P_{Y} (x, y) := \PP (Y_{n+1} = y | Y_{n} = x)$ denote the one-step homogeneous transition probability matrix of  $Y_n$, as given in (\ref{WF-D}). Employing the Markov property once more, we obtain
\begin{align*}
& \EE_x \lt[ f \lt( \hat{X}_{n} \big( 1, Y_{1} \big) \rt) \rt] = \sum_{y_1 \in I_{2N}} P_Y (x, y_1) \EE_x \lt[ f \Big( \hat{X}_{n} \big( 1, y_1 \big) \Big) \rt]\\
 =& \sum_{y_1 \in I_{2N}} P_Y (x, y_1) \EE_{y_1} \lt[ f \Big( \hat{X}_{n} \big( 2, \hat{X}_2 (1, y_1) \big) \Big) \rt] \\
=& \sum_{y_1 \in I_{2N}} P_Y (x, y_1) \lt\{\EE_{y_1} \lt[ f \Big( \hat{X}_{n} \big( 2, \hat{X}_2 (1, y_1) \big) \Big) \rt] - \EE_{y_1} \lt[ f \Big( \hat{X}_{n} \big( 2, Y_2 \big) \Big) \rt]\rt\}\\
& + \sum_{y_1 \in I_{2N}} P_Y (x, y_1) \EE_{y_1} \lt[ f \Big( \hat{X}_{n} \big( 2, Y_2 \big) \Big) \rt].
\end{align*}
Therefore, by mathematical induction, we obtain that
\begin{align}\lb{thm-1-1}
\EE_x \lt[ f (X_{n}) \rt] & = \EE_x \lt[ f \Big( \hat{X}_n \big( 1, \hat{X}_1 (0, Y_0) \big) \Big) \rt] - \EE_x \lt[ f \Big( \hat{X}_n (1, Y_1) \Big) \rt] \notag \\
& \qd + \sum_{j=1}^{n-1} \lt\{ \sum_{y_j \in I_{2N}} P_Y^j (x, y_j) \lt[ \EE_{y_j} \lt[ f \Big( \hat{X}_n \big( j+1, \hat{X}_{j+1} (j, y_j) \big) \Big) \rt]\right.\right.\nonumber\\
 &\left.\left.\qquad\qquad\qquad\qquad\qquad\qquad- \EE_{y_j} \lt[ f \Big( \hat{X}_n (j+1, Y_{j+1}) \Big) \rt] \rt] \rt\} \notag \\
& \qd + \sum_{y_{n-1} \in I_{2N}} P_Y^{n-1} (x, y_{n-1}) \EE_{y_{n-1}} \lt[ f \Big( \hat{X}_n (n, Y_n) \Big) \rt].
\end{align}
By (\ref{relat1}), we have $X_n (n, Y_n) \overset{\d}{=} Y_n$, which implies that
$$
\sum_{y_{n-1}} P_Y^{n-1} (x, y_{n-1}) \EE_{y_{n-1}} \lt[ f \Big( \hat{X}_n (n, Y_n) \Big) \rt] = \EE_x \lt[ f (Y_n)\rt].
$$
With the supplementary definition of $P_Y^0 (x, y_0) := \dz_{x} (y_0)$, we establish that (\ref{thm-1-1}) is equivalent to following
\begin{align}\lb{thm-1-2}
\EE_x \lt[ f (X_{n}) \rt] - \EE_x \lt[ f( Y_{n}) \rt]
=& \sum_{j=1}^{n} \lt\{ \sum_{y_{j-1} \in I_{2N}} P_Y^{j-1} (x, y_{j-1}) \left[ \EE_{y_{j-1}} \lt[ f \Big( \hat{X}_n \big( j, \hat{X}_{j} (j-1, y_{j-1}) \big) \Big) \rt]\right.\right.\nonumber \\
& \left.\left.\qquad\qquad\qquad\qquad\qquad\qquad-\EE_{y_{j-1}} \lt[ f \Big( \hat{X}_n (j, Y_{j}) \Big) \rt] \right] \rt\}.
\end{align}

Next, we give the bounds of the right-hand side of (\ref{thm-1-2}). Similar to the proof of Lemma \ref{regular}, we denote
$$
F_t (x) := \EE_x \lt[ f(X_t)\rt] = P_t f (x), \qd \text{for all $t \geq 0$, $f \in \mathcal{C}_b^4 (I)$ and $x \in I_{2N} \subset I$. }
$$
In the above definition, the $P_t$ is the semi-group of process $X_t$. According to \cite[Theorem 1]{Eth76}, it holds that $F_t \in \mathcal{C}_b^4 (I)$ for all $t \geq 0$. By the Markov property of $X_t$, for each $0 \leq j \leq n$, we obtain
$$
F_{n - j} (x) = \EE_x \lt[ f (X_{n-j}) \rt] = \EE \lt[ f \big( \hat{X}_n (j, x) \big) \big| \hat{X}_j (j, x) = x \rt], \qd \forall x \in I_{2N}.
$$
Hence, we have
\begin{align*}
&\EE_{y_{j-1}} \lt[ f \Big( \hat{X}_n \big( j, \hat{X}_{j} (j-1, y_{j-1}) \big) \Big) \rt] - \EE_{y_{j-1}} \lt[ f \Big( \hat{X}_n (j, Y_{j}) \Big) \rt]\\
=& \EE_{y_{j-1}} \lt[ F_{n-j} \lt( \hat{X}_1 (0, y_{j-1}) \rt) \rt] - \EE_{y_{j-1}} \lt[ F_{n-j} \lt( Y_{j} \rt) \rt].
\end{align*}
As per  Lemma \ref{compare} (i), we deduce
\begin{align}\lb{thm-1-3}
\lt| \EE_{y_{j-1}} \lt[ F_{n-j} \lt( \hat{X}_1 (0, y_{j-1}) \rt) \rt] - \EE_{y_{j-1}} \lt[ F_{n-j} \lt( Y_{j} \rt) \rt] \rt|
\leq& \sum_{m=1}^4 C_m \lt( \sup_{y \in I} \lt\| \nabla^m F_{n-j} (y) \rt\|_2 \rt),
\end{align}
and each item of (\ref{thm-1-3}) can be estimated based on Lemma \ref{regular}, i.e., for $m \in \{1, \ldots, 4\}$
\begin{align}\lb{thm-1-4}
\sup_{y \in I} \lt\| \nabla^m F_{n-j} (y) \rt\|_2 =& \sup_{y \in I} \Big\| \nabla^m P_{n-j} f (y) \Big\|_2 \nonumber\\
\leq& \Big\| \sup_{y \in I} \nabla^m f (y) \Big\|_2 \exp \lt( - (n-j) \min_{1 \leq k_1, \ldots, k_m \leq r-1} \lz_{k_1, \ldots, k_m} \rt).
\end{align}
Hence, combining (\ref{thm-1-2}), (\ref{thm-1-3}) and (\ref{thm-1-4}), we arrive at
\begin{align*}
& \lt| \EE_x \lt[ f (X_{n}) \rt] - \EE_x \lt[ f( Y_{n}) \rt] \rt|\\
\leq& \sum_{j=1}^n \lt\{ \sum_{y_{j-1} \in I_{2N}} P^{j-1}_Y (x, y_{j-1}) \lt[ \sum_{m=1}^{4} C_m \lt( \sup_{y \in I} \lt\| \nabla^m F_{n-j} (y) \rt\|_2 \rt) \rt] \rt\} \\
\leq& \sum_{j=1}^n \sum_{m=1}^4 C_m \Big\| \sup_{y \in I} \nabla^m f (y) \Big\|_2 \exp \lt( - (n-j) \min_{1 \leq k_1, \ldots, k_m \leq r-1} \lz_{k_1, \ldots, k_m} \rt)  \\
\leq& \sum_{m=1}^4 C_m \Big\| \sup_{y \in I} \nabla^m f (y) \Big\|_2 \lt[ \frac{1 - \exp \lt(-n \min_{1 \leq k_1, \ldots, k_m \leq r-1} \lz_{k_1, \ldots, k_m} \rt)}{1- \exp \lt( - \min_{1 \leq k_1, \ldots, k_m \leq r-1} \lz_{k_1, \ldots, k_m} \rt)} \rt].
\end{align*}
Hence, the proof is completed. \deprf

\subsection{Proof of Corollary \ref{cor}.}
By using (\ref{2semi}) and referring to Lemma \ref{compare} (ii), applying a similar argument as employed in proving Theorem \ref{main1}, we arrive at the desired result.
\deprf

\begin{appendix}

\section{the existence and uniqueness of the solution}\label{AppendixA}

\subsection{The strong solution when $r=2$}

When $r=2$, we can prove the uniqueness of the strong solution. In the following, we present a version of the Yamada-Watanabe criterion cited from \cite[Theorem 1]{YW71}.
\begin{thm}\label{YaWa}
Let
\begin{align}\label{ASDE}
dX_t = b(X_t)dt + \sigma(X_t)dB_t.
\end{align}
Assume that
\begin{enumerate}
\item[(1)]
there exists a positive increasing function $\rho(u)$, $u \in (0,\infty)$ such that
\[
|\sigma(x) - \sigma(y)| \le \rho(|x-y|), \quad \forall x,y \in \mathbb{R}
\]
\[
\text{and} \quad \int_{0+} \frac{1}{\rho^2(u)}du =+\infty,
\]
\item[(2)]there exists a posotive increasing concave function $\kappa(u)$, $u \in (0,\infty)$, such that
\[
|b(x) - b(y)| \le \kappa (|x-y|), \quad \forall x,y \in \mathbb{}
\]
\[
\text{and} \quad  \int_{0+}\frac{1}{\kappa(u)}du =+\infty.
\]
\end{enumerate}
Then the pathwise uniqueness holds for \eqref{ASDE}.
\end{thm}
Now we are ready to prove the pathwise uniqueness of solution to the SDE \eqref{WF-C}.

\begin{thm}
The SDE \eqref{WF-C}
has a unique strong solution $\{X_t, t\in \RR_+\}$ such that $X_t \in [0,1]$ for all $t\in \RR_+$ almost surely.
\end{thm}
\begin{proof}
Let $\sigma(x) = \sqrt{x(1-x)}$ and  $b(x) = -\mu_1 x +\mu_2 (1-x)$. Then, for any $x,y \in [0,1]$, we have
	\begin{align*}\label{sqrt}
	&|\sigma(x) - \sigma(y)|^2=\left[\sqrt{x(1-x)}-\sqrt{y(1-y)}\right]^2   \nonumber \\
       =& \left[\sqrt{x(1-x)}-\sqrt{y(1-x)}+\sqrt{y(1-x)}-\sqrt{y(1-y)}\right]^2  \nonumber \\
	\le& 2(\sqrt{x(1-x)}-\sqrt{y(1-x)})^2 + 2(\sqrt{y(1-x)}-\sqrt{y(1-y)})^2 \nonumber \\
	\le& 2 |1-x|(\sqrt{x}-\sqrt{y})^2 + 2 |y|(\sqrt{1-x}-\sqrt{1-y})^2 \nonumber \\
	\le& 2 |1-x|(\sqrt{x}-\sqrt{y})(\sqrt{x}+\sqrt{y}) + 2 |y|(\sqrt{1-x}-\sqrt{1-y})(\sqrt{1-x}+\sqrt{1-y}) \nonumber \\
	\le& 2|1-x||x-y|+2|y||x-y|\nonumber \\
       \le& 2|x-y|.
	\end{align*}
Choose $\rho(u) = \sqrt{u/N}$, then $\rho$ is an increasing function on $\RR_+$ such that $|\sigma (x) - \sigma (y)| \leq \rho (|x-y|)$ for all $x,y \in [0,1]$,  and
\[
\int_{0+} \frac{1}{\rho^2(u)}du = N\int_{0+} \frac{1}{u}du =\infty.
\]
On the other hand, we can also choose $\kappa(u) = (|\mu_{1}| + |\mu_{2}|)u$. Then $\kappa$ is a increasing  function such that $|b(x) - b(y)| \leq \kappa (|x-y|)$, and
\[
\int_{0+} \frac{1}{\kappa(u)}du = (|\mu_1| + |\mu_2|)^{-1} \int_{0+} \frac{1}{u}du =\infty.
\]
By applying Theorem \ref{YaWa}, we deduce the pathwise uniqueness for SDE \eqref{WF-C}  follows immediately, which complete the proof.
\end{proof}

\subsection{The martingale solution when $r\geq2$}

We first give the existence of martingale solution.

For each $t\geq0$, define $X(t):\mathcal{C}([0,\infty),I)\rightarrow I$ by $X(t)(\omega)=\omega(t)$, and let $\mathcal{F}_{t}$ and $\mathcal{F}$ be the $\sigma$-fields generated by $\{X(s):0\leq s\leq t\}$ and $\{X(s):0\leq s<\infty\}$, respectively. Define the operator $\mathcal{L}$ by \eqref{oper}, that is,
\be\lb{Aoper}
\mathcal{L} = \frac{1}{4N} \sum_{i, j = 1}^{r-1} a_{ij} (x) \frac{\p^2}{\p x_i \p x_j} + \sum_{i=1}^{r-1} b_i (x) \frac{\p}{\p x_i}, \qd \forall x \in I,
\de
where matrix $A(x) = \lt( a_{ij} (x) \rt)_{1 \leq i, j \leq r-1} \in \RR^{(r-1) \times (r-1)}$ and vector $b(x) = \lt( b_i (x) \rt)_{1 \leq i \leq r-1} \in \RR^{r-1}$ are given as
\begin{align*}
a_{ij} (x) &= x_i (\dz_{ij} - x_j), \qd x \in I, \\
b_i (x) &= - x_i \sum_{j=1}^r u_{ij} + \sum_{j=1}^r x_j u_{ji}, \qd x \in I,
\end{align*}
 and $x_{r}=1-\sum_{i=1}^{r-1}x_{i}\in[0,1]$. Let $x\in I$, a probability measure $\mathbb{P}$ on $\left(\mathcal{C}([0,\infty),I),\mathcal{F}\right)$ is said to solve the martingale problem for $\mathcal{L}$ starting at $x$ if $\mathbb{P}\left(X(0)=x\right)=1$ and
$$
\left\{f\left(X(t)\right)-\int_{0}^{t}\mathcal{L}f\left(X(s)\right)ds,\mathcal{F}_{t}:t\geq0\right\}
$$
is a $\mathbb{P}$-martingale for each $f\in\mathcal{C}^{2}(I)$.

\begin{thm}
Define the operator $\mathcal{L}$ by \eqref{Aoper}. Then, for each $x\in I$, there is only one solution to the martingale problem for $\mathcal{L}$ starting at $x$.
\end{thm}

\begin{proof}
{\bf Existence:} Define the retract $\rho:\mathbb{R}^{r-1}\rightarrow I$ by setting $\rho(x)=y$ if $\inf\{|z-x|:z\in I\}=|y-x|$, and for $i=1,\cdots,r$, define $H_{i}\subseteq\mathbb{R}^{r-1},K_{i}\subseteq\mathbb{R}^{r-1}$ and $\nu^{i}\in\mathbb{R}^{r-1}$ by
\begin{align*}
H_{i}=&\left\{x\in\mathbb{R}^{r-1}:x_{i}\geq0\right\}, \qquad i\neq r,\\
H_{r}=&\left\{x\in\mathbb{R}^{r-1}:\sum_{j=1}^{r-1}x_{j}\leq 1\right\},\\
I_{i}=&I\cap\partial H_{i},\\
\nu_{j}^{i}=&\delta_{ij}-\delta_{ir}.
\end{align*}
Fix $x\in I$. Since both $\frac{A}{2N}\circ\rho$ and $b\circ\rho$ are bounded and continuous, the stochastic integral equation
\begin{align*}
X_{t}=x+\int_{0}^{t}\left(\frac{A}{2N}\circ\rho\right)^{\frac{1}{2}}(X_{s})dB_{s}+\int_{0}^{t}\left(b\circ\rho\right)(X_{s})ds
\end{align*}
has a solution $\left(\Omega,\hat{\mathcal{F}},\hat{\mathcal{F}}_{t},B_{t},X_{t},\hat{\mathbb{P}}\right)$.

Then by reproducing the same argument as in \cite[Section 2]{Eth76}, the existence of the martingale problem reduces to the validity of the following conditions
 \begin{align}\label{condition}
 \left\langle\nu^{i},\frac{A(x)}{2N}\nu^{i}\right\rangle=0\quad \text{and}\quad \left\langle\nu^{i},b\right\rangle \ge 0,
 \end{align}
 on $I_i$ for all $i=1,\cdots, r$.



Now, it remains to verify \eqref{condition} by studying the following two cases.

\underline{Case $1$.} For $i=1,\cdots,r-1$, we have $\nu_{j}^{i}=\delta_{ij}$ for $j=1,\cdots,r-1$. Then for any $x\in I_{i}$, we have
\begin{align*}
\left\langle\nu^{i},\frac{A(x)}{2N}\nu^{i}\right\rangle=\frac{1}{2N}\sum_{j,k=1}^{r-1}a_{jk}(x)\nu_{j}^{i}\nu_{k}^{i}=\frac{1}{2N}a_{ii}(x)=0
\end{align*}
and
\begin{align*}
\left\langle\nu^{i},b\right\rangle=\sum_{j=1}^{r-1}\nu_{j}^{i}b_{j}(x)=b_{i}(x)=-x_i\sum_{j=1}^r u_{ij}+\sum_{j=1}^{r}x_{j}u_{ji}=\sum_{j=1}^{r}x_{j}u_{ji}\geq0,
\end{align*}
where the last equality holds because $x_i=0$ when $x\in I_{i}, i=1,\cdots,r-1$.

\underline{Case $2$.} For $i=r$, we have $\nu_{j}^{r}=-1$ for $j=1,\cdots,r-1$. Then for any $x\in I_{r}$, we have
\begin{align*}
\left\langle\nu^{r},\frac{A(x)}{2N}\nu^{r}\right\rangle=\frac{1}{2N}\sum_{j,k=1}^{r-1}a_{jk}(x)\nu_{j}^{r}\nu_{k}^{r}
=&\frac{1}{2N}\sum_{j,k=1}^{r-1}x_{j}(\delta_{jk}-x_{k})\\
=&\frac{1}{2N}\bigg(\sum_{j=1}^{r-1}x_{j}-\sum_{j=1}^{r-1}x_{j}\sum_{k=1}^{r-1}x_{k}\bigg)=0,
\end{align*}
where the last equality is due to the fact that $\sum_{i=1}^{r-1}x_{i}=1$ on $I_r$.
In addition,  note  that $x_{r}=1-\sum_{i=1}^{r-1}x_{i}=0$, we have
\begin{align*}
\left\langle\nu^{r},b\right\rangle=&\sum_{i=1}^{r-1}\nu_{i}^{r}b_{i}(x)=-\sum_{i=1}^{r-1}b_{i}(x)\\
=&\sum_{i=1}^{r-1}x_i \sum_{j=1}^r u_{ij}-\sum_{i=1}^{r-1}\sum_{j=1}^r x_j u_{ji}\\
=&\sum_{i=1}^{r-1}x_i \sum_{j=1}^r u_{ij}-\sum_{i=1}^{r-1}\sum_{j=1}^{r-1} x_j u_{ji}\\
=&\sum_{i=1}^{r-1}\sum_{j=1}^{r-1}x_i  u_{ij}-\sum_{i=1}^{r-1}\sum_{j=1}^{r-1} x_j u_{ji}+\sum_{i=1}^{r-1}x_{i}u_{ir}=\sum_{i=1}^{r-1}x_{i}u_{ir}\geq 0.
\end{align*}
In other words, \eqref{condition} holds for all $i = 1,\dots, r$.

\noindent {\bf Uniqueness:} By the same arguments as the proof of \cite[Theorem 1]{Eth76}, the uniqueness of the martingale solution follows immediately from Lemma \ref{regular}.

\end{proof}

\end{appendix}

\noindent
\textbf{Acknowledgement.} The research of L.\ Xu is supported by National Natural Science Foundation of China No. 12071499, The Science and Technology Development Fund (FDCT) of Macau S.A.R. FDCT 0074/2023/RIA2, and University of Macau grant MYRG-GRG2023-00088-FST. J. Xiong was supported by National Natural Science Foundation of China grants 12326368, National Key R\&D Program of China grant 2022YFA1006102. J. Zheng has been supported by National Natural Science Foundation of China grant (Grant No. 11901598) and Guangdong Characteristic Innovation Project No.2023KTSCX163. P.\ Chen has been supported in part by National Natural Science Foundation of China grant (Grant No. 12301176) and Natural Science Foundation of Jiangsu Province grant (Grant No. BK20220867).

\bibliographystyle{amsplain}

\begin{thebibliography}{99}\frenchspacing

\bibitem{CSZ16} Caravenna, F., Sun, R. and Zygouras, N. (2016). Polynomial chaos and scaling limits of disordered systems. {\it Journal of the European Mathematical Society}, {\bf 19}(1), 1-65.

\bibitem{Cha06} Chatterjee, S. (2006). A generalization of the Lindeberg principle. {\it The Annals of Probability}, {\bf 34}(6), 2061-2076.

\bibitem{ChMe08} Chatterjee, S. and Meckes E. (2008). Multivariate normal approximation using exchangeable pairs. {\it Alea}, {\bf 4}, 257-283.


\bibitem{CLX22} Chen, P., Lu, J. and Xu, L. (2022). Approximation to stochastic variance reduced gradient Langevin dynamics by stochastic delay differential equations. {\it Applied Mathematics \& Optimization}, {\bf 85}(2), 1-40.

\bibitem{CSX20} Chen, P., Shao, Q. M. and Xu, L. (2020). A probability approximation framework: Markov process approach. {\it Annals of Applied Probability}, {\bf 33}(2), 1619-1659.

\bibitem{CX19} Chen, P. and Xu, L. (2019). Approximation to stable law by the Lindeberg principle. {\it Journal of Mathematical Analysis and Applications}, {\bf 480}(2), 123338.

\bibitem{CCK14} Chernozhukov, V., Chetverikov, D. and Kato, K. (2014). Gaussian approximation of suprema of empirical processes. {\it The Annals of Statistics}, {\bf 42}(4), 1564-1597.

\bibitem{CCK17}Chernozhukov, V., Chetverikov, D. and Kato, K. (2017). Central limit theorems and bootstrap in high dimensions. {\it The Annals of Probability}, {\bf 45}(4), 2309-2352.

\bibitem{D17} Dawson, D. A. (2017). Introductory Lectures on Stochastic Population Systems. {\it preprint in arXiv: 1705.03781v1}.

\bibitem{D69} Dieudonn\'{e}, J. (1969). Foundations of Modern Analysis. {\it Academic press, Boston, MA, MR 0349288}.

\bibitem{DNN11} Dudley, R. M., Norvai\v{s}a, R. and Norvai\v{s}a, R. (2011). Concrete functional calculus. {\it New York: Springer}.

\bibitem{Eth76} Ethier, S. N. (1976). A class of degenerate diffusion processes occurring in population genetics. {\it Communications on Pure and Applied Mathematics}, 29(5), 483-493.

\bibitem{EK93} Ethier, S. N. and Kurtz, T. G. (1993). Fleming-Viot processes in population genetics. {\it SIAM Journal on Control and Optimization}, {\bf 31}(2), 345-386.

\bibitem{EK09} Ethier, S. N. and Kurtz, T. G. (2009). Markov processes: characterization and convergence. Vol. 282. {\it John Wiley and Sons Ltd, Chichester}.

\bibitem{EN77} Ethier, S. N. and Norman, M. F. (1977). {Error estimate for the diffusion approximation of the Wright-Fisher model.} \textit{Proc. Natl. Acad. Sci.}~\textbf{74}(11), 5096--5098.

\bibitem{FSX19} Fang, X., Shao, Q. M. and Xu, L. (2019). Multivariate approximations in Wasserstein distance by Stein's method and Bismut's formula. {\it Probability Theory and Related Fields.} {\bf 174}(3), 945-979.

\bibitem{F30} Fisher, R. A. (1930). The genetical theory of natural selection: a complete variorum edition. {\it Oxford University Press}.

\bibitem{Fri75} Friedman, A. (1975). Stochastic differential equations and applications. Vol 1, {\it Academic Press, New York}.

\bibitem{GHK21}Garcia-Pareja, C., Hult, H. and Koski, T. (2021). Exact simulation of coupled Wright-Fisher diffusions. {\it Advances in Applied Probability}, {\bf 53}(4), 923-950.

\bibitem{HJT17} Hofrichter, J., Jost, J. and Tran, T. D. (2017). Information Geometry and Population Genetics. {\it Springer International Publishing}.
    
\bibitem{KM11} Korada, S. B. and Montanari, A. (2011). Applications of the Lindeberg principle in communications and statistical learning. {\it IEEE transactions on information theory}, {\bf 57}(4), 2440-2450.

\bibitem{L96} Lieberman, G. M. (1996). Second Order Parabolic Differential Equations. World Scientific Publishing Co. Pte. Ltd., Singapore.

\bibitem{Lin22} Lindeberg, J. W. (1922). Eine neue Herleitung des Exponentialgesetzes in der Wahrscheinlichkeitsrechnung. {\it Mathematische Zeitschrift}, {\bf 15}(1), 211-225.

\bibitem{Nor72} Norman, M. F. (1972). Markov processes and learning models (Vol. 84). {\it Academic Press, New York}.

\bibitem{TV11} Tao, T. and Vu, V. (2011). Random matrices: universality of local eigenvalue statistics. {\it Acta mathematica}, {\bf 206}(1), 127-204.

\bibitem{T12} Tran, T. D. (2012). Information geometry and the Wright-Fisher model of mathematical population genetics. {\it Ph.D. thesis, University of Leipzig}.

\bibitem{Wan23} Wang, F. Y. (2023). Wasserstein Convergence Rate for Empirical Measures of Markov Processes. arxiv preprint arxiv:2309.04674.

\bibitem{W31} Wright, S. (1931). Evolution in Mendelian populations. {\it Genetics.} 16(2), 97--159.

\bibitem{YW71} Yamada, T., Watanabe, S. (1971). On the uniqueness of solutions of stochastic differential equations. {\it J. Math. Kyoto Univ.},  {\bf 11}, 155-167.
\end{thebibliography}

\end{document}